\documentclass[11pt]{amsart}
\usepackage{amssymb,amsmath,amsthm,latexsym}
\usepackage{amssymb,graphicx}
\newtheorem{theorem}{Theorem}

\newtheorem{corollary}[theorem]{Corollary}

\newtheorem{lemma}[theorem]{Lemma}

\newtheorem{question}[theorem]{Question}

\begin{document}

\title{A non-LEA Sofic Group} 

\author{Aditi Kar \and Nikolay Nikolov}
\noindent \address{ A. Kar, University of Southampton, SO17 1BJ, UK, a.kar@soton.ac.uk \newline N. Nikolov, University of Oxford, OX2 6GG, UK. nikolov@maths.ox.ac.uk.} 
\begin{abstract}
We describe elementary examples of finitely presented sofic groups which are not residually amenable (and thus not initially subamenable or LEA, for short). We ask if an amalgam of two amenable groups over a finite subgroup is residually amenable and answer this positively for some special cases, including countable locally finite groups, residually nilpotent groups and others.

\end{abstract}
\maketitle

Gromov \cite{gromov} defined the class of sofic groups as a generalization of residually finite and of amenable groups, see \cite{cap} for an introduction. The question of the existence of a non-sofic group remains open.




Two related classes of groups defined in \cite{gromov} are the LEF groups and LEA groups. A group $G$ is LEF (locally embeddable in finite) if for every finite subset $F \subset G$, there exists a partial monomorphism of $F$ into a finite group $H$, meaning there exists an injection $f: F \rightarrow H$ such that if $x,y,xy \in F$ then $f(xy)=f(x)f(y)$. Similarly, a group is LEA (locally embeddable in amenable, also called \emph{initially subamenable}) if every finite subset of $G$ supports a partial monomorphism to an amenable group. It is not hard to show that amongst finitely presented groups, the class of LEF groups coincides with the residually finite groups and the class of LEA groups coincides with the residually amenable groups. 

Gromov's question whether every sofic group is LEA was answered negatively by Cornulier \cite{Yves}. Cornulier's example is an extension of an LEF group by an amenable quotient and hence, the LEF normal subgroup is co-amenable. For an introduction to co-amenability, we refer the reader to \cite{MPopa}.  

It was recently proved that if $A$ and $B$ are sofic groups and $C$ is an amenable subgroup of both $A$ and $B$, then the free product with amalgamation $G=A*_CB$ is also sofic, see \cite{CollinsDykema}, \cite{ElekSzabo} and \cite{Paunescu}. This theorem allows us to produce further examples of finitely presented sofic groups which are not residually amenable or equivalently LEA. The examples actually show that the class of LEA groups is not closed under taking free products with infinite cyclic amalgamations. 

Consider $\mathrm{SL}_n(\mathbb{Z}[\frac{1}{p}])$ for $n \geq 3$ and $p$, a prime. For $x \in \mathbb Z [\frac{1}{p}]$ denote $z(x) = \left(
\begin{array}{ccc} 
1 & x & 0 \\ 
0 & 1 &  0 \\
0 & 0 & \mathrm{I}_{n-2} 
\end{array}
\right)$ and  $Z$ be the infinite cyclic subgroup generated by the element 
$z=z(1)$.

\begin{theorem} \label{main}
The amalgam $G= \mathrm{SL}_n(\mathbb{Z}[\frac{1}{p}]) *_Z \mathrm{SL}_n(\mathbb{Z}[\frac{1}{p}])$ is sofic but not LEA. In fact $G$ does not have a co-amenable LEA subgroup.
\end{theorem} 

\begin{proof} As we noted earlier, soficity of $G$ follows from \cite{CollinsDykema}. Now, suppose for the sake of contradiction that $H$ is a co-amenable LEA subgroup of $G$. We denote the two copies of $\mathrm{SL}_n(\mathbb{Z}[\frac{1}{p}])$ as $\Sigma_1$ and $\Sigma_2$ with two isomorphisms $f_i: \mathrm{SL}_n(\mathbb{Z}[\frac{1}{p}]) \rightarrow \Sigma_i$ ($i=1,2$). We have $f_1(z)=f_2(z)$ and will identify the cyclic group $Z$ with its image $f_1(Z)=f_2(Z)$ in $G$.
Define $Y= \{ y(x) \ | \ x  \in \mathbb Z [\frac{1}{p}] \}$ and let $Y_i= f_i(Y)$, a subgroup of $\Sigma_i$ isomorphic to $(\mathbb Z[\frac{1}{p}], +)$.  Thus $Y_1 \cap Y_2=Z$.

We claim that some $G$-conjugate of $H$ intersects both $\Sigma_1$ and $\Sigma_2$ in subgroups of finite index. Let $X=G/H$ be the set of left cosets of $H$ in $G$. The co-amenability of $H$ in $G$ is equivalent to the amenability of the action of $G$ on $X$, that is, the existence of a finitely additive $G$-invariant mean $\mu$ on all subsets of $X$ \cite{MPopa}. Since $\Sigma_i$ is a lattice in $\mathrm{SL}_n(\mathbb R) \times \mathrm{SL}_n(\mathbb Q_p)$  it has property $(T)$ (see for example \cite[Theorem 1.4.15]{BHV}). Therefore every amenable action of $\Sigma_i$ must have a finite orbit \cite[Lemma 4.2]{GM}. Let $X_{i, \infty} \subset X$ be the union of all the infinite orbits of $\Sigma_i$ on $X$. We must have that $\mu(X_{i, \infty})=0$, otherwise we can normalize $\mu$ on $X_{i, \infty}$ to a $\Sigma_i$-invariant mean where $\Sigma_i$  acts on $X_{i,\infty}$ without finite orbits. Therefore $ \mu(X_{1, \infty} \cup X_{2,\infty})=0$ and we can take $t \in X \backslash (X_{1, \infty} \cup X_{2,\infty})$. The stabilizer of $t$ in $G$ is the required conjugate of $H$.

By replacing $H$ with $\mathrm{Stab}_G(t)$ we may assume from now on that $M_i= \Sigma_i \cap H$ has finite index in $\Sigma_i$ for $i=1,2$. In particular $H \cap Y_i$ is a finite index subgroup of $Y_i$. 
Note that $Y_i/Z$ is a divisible group and so $Y_i= (Y_i \cap H)Z$ giving that $[Y_1: (Y_1 \cap H)]= [Z: (Z \cap H)]=[Y_2: (Y_2 \cap H)]=a$, say. The integer $a$ is coprime to $p$ and $H \cap Y_i= f_i \circ y (a \mathbb Z [\frac{1}{p}])$. Define $\beta_i=f_i(y(\frac{a}{p}))$, thus $\beta_i \in (H \cap Y_i) \backslash Z$ while $\beta_1^p=\beta_2^p \in Z \cap H$.

Note that each $M_i=\Sigma_i \cap H$ is a finitely presented group. Therefore there exists a finite set $F \subset M_1 \cup M_2$ such that every partial monomorphism $h$ from $F \subset H$ into any group $A$ extends to a map $h: M_1 \cup M_2 \rightarrow A$,  such that $h|_{M_i}$ is a group homomorphism. Now assume that $A$ is amenable. Since $M_1, M_2$ have property $(T)$ their images $h(M_i)$ in $A$ are finite.
In particular the groups $h(H \cap Y_1)$ and $h( H \cap Y_2)$ are finite and $p$-divisible and so must have size coprime to $p$. Every finite image of $H \cap Y_i$ coincides with the image of its subgroup $Z \cap Y_i$ and so $h(H \cap Y_1)= h(H \cap Z)= h(H \cap Y_2) $ is a finite group $L$ with $gcd(p, |L|)=1$. Note that $h(\beta_1)^p=h(\beta_2)^p$. Since every element of $L$ has a unique $p$-th root it follows that $h(\beta_1)=h(\beta_2)$. However $\beta_1 \not = \beta_2$ in $H$. We deduce that there cannot be a partial monomorphism from $F \cup \{\beta_1, \beta_2\}$ into any amenable group, and so $H$ is not LEA.
 \end{proof}
\subsection*{Residually amenable groups}
Theorem \ref{main} shows that a free product with amalgamation over $\mathbb Z$ of two residually amenable groups need not be residually amenable. On the other hand it is not difficult to show that a free product of residually amenable groups is residually amenable, for a proof see  \cite{berlai}. A natural question is whether this result extends to amalgamations over finite subgroups:

\begin{question}
\label{q1} 
Is the class of residually amenable groups closed under free products with amalgamation over finite subgroups?
\end{question}

To show that a free product $G$ of residually amenable groups $A,B$ amalgamating a finite subgroup is residually amenable, it is sufficient to consider the case when $A$ and $B$ are both amenable. Let $g$ be a non-trivial element of $G$ and let $S=\{a_1,\ldots,a_k\}$ and $T:=\{b_1,\ldots, b_n\}$ be the non-trivial syllables of $g$. Then, there exist homomorphisms $\alpha, \beta$ from $A, B$ respectively to amenable groups $A',B'$, which are injective on $C \cup S$  and $C\cup T$. This clearly allows one to define a homomorphism from $G$ to $A'*_C B'$ such that the image of $g$ is non-trivial. 

Another question related to Question \ref{q1} is the following. 

For a class of groups $\mathcal L$ we say that $\mathcal L$ \emph{admits amalgamations over finite subgroups} when the following condition holds: For all $A, B \in \mathcal L$ and a finite group $C$ having two injective homomorphisms $\phi_A, \phi_B$ to $A$ and $B$ respectively, there exists a group $H \in \mathcal L$ and two injective homomorphisms $\psi_A, \psi_B$ from $A, B$ respectively to $H$ such that $\psi_A \circ \phi_A=\psi_B \circ \phi_B$.

\begin{question} \label{q2} Does the class of amenable groups admit amalgamations over finite subgroups?
\end{question}

The authors don't know the answer to Question \ref{q2} even in the case when $A$ and $B$ are solvable groups. Note that if Question \ref{q2} has positive answer then there is a homomorphism $f:A*_CB \rightarrow H$ such that the the kernel of $f$ is free. Hence $A *_C B$ is residually amenable, which easily implies that Question \ref{q1} has a positive answer.

We can answer questions \ref{q1} and \ref{q2} affirmatively in the special case of locally finite groups.

\begin{theorem}The class of countable locally finite groups admits amalgamations over finite subgroups. As a consequence, the amalgam of two countable locally finite groups over a finite subgroup is residually amenable.
\end{theorem}

\begin{proof}  Let $A$ and $B$ be two countable locally finite groups and let $C$ be a finite group with two injective homomorphisms $\phi_A, \phi_B$ to $A$ and $B$ respectively. Let $A$ be the direct limit of the finite groups $\{A_{j}\}_{j=1}^\infty$ with injective homomorphisms $f_{A,i}: A_{i} \rightarrow A_{i+1}$ and denote by $\{B_j\}_{j=1}^\infty$ and $f_{B,i}: B_{i} \rightarrow B_{i+1}$ the corresponding direct limit for $B$. Since $C$ is finite we may assume that $\phi_A (C) \subset A_{1}$ and $\phi_B (C) \subset B_{1}$. Now recall the following lemma

\begin{lemma}[Lemma II.2.6.10 \cite{se}] \label{serre} The class of finite groups admits amalgamations over finite subgroups.
\end{lemma} 
We apply Lemma \ref{serre}  to $A_{1},B_{1},C,\phi_A,\phi_B$ and obtain a finite group $H_1$ with injections $t_{A,1},t_{B,1}$ from $A_1$, respectively $B_1$ to $H_1$ such that $t_{A,1} \phi_A= t_{B,1} \phi_B$. 

Next we apply Lemma \ref{serre} to $H_1,A_2,A_1, t_{A,1},f_{A,1}$ and obtain a finite group $H_2$ with injections $d_1: H_1 \rightarrow H_2$ and $t_{A,2}: A_2 \rightarrow H_2$ which agree on $A_1$.

The next application of Lemma \ref{serre} to $H_2,B_2,B_1, d_1 t_{B,1}, f_{B,1}$ gives a finite group $H_3$ and injections $d_2: H_2 \rightarrow H_3$ and $t_{B,2}: B_2 \rightarrow H_3$ which agree on $B_1$. 
Continuing in the same way by induction we get a directed system $(H_j)$ of finite groups with injections $d_i :H_i\rightarrow H_{i+1}$ such that $H_{2i+1}$ contains an embedded copy of $A_i$ and $B_i$ with intersection containing an isomorphic copy of $C$. Let $H$ be the direct limit of $(H_j)$. By construction there are induced  monomorphisms $\psi_A$ and $\psi_B$ from $A$ and $B$ to $H$  as required.

The final part follows by observing that the homomorphism $\psi_A * \psi_B : A*_CB \rightarrow H$ has kernel which is a free group $K$. Now if $K^{(n)}$ is the $n$-th term of the derived series of $K$ we have that $\frac{A*_C B} {K^{(n)}}$ is an amenable group for each $n \in \mathbb N$ and therefore $A*_CB$ is residually amenable.
\end{proof}

Recall that the $FC$-centre of $G$ is defined to be the union of the finite conjugacy classes in $G$. This is a characteristic subgroup of $G$.
\begin{theorem}Let $A, B$ be residually amenable groups and let $C$ be a finite group such that $C$ is contained in the $FC$-centre of both $A$ and $B$. Then, $A*_C B$ is residually amenable.
\end{theorem}

\begin{proof} 

Note first that the normal closures $N$, (resp. $M$) of $C$ in $A$ (resp. $B$) must be finite. Indeed $N$ is generated by the finitely many conjugates of $C$ in $A$, implying that the centre of $N$ has finite index in $N$ and therefore $N$ must be finite because it is generated by finitely many elements of finite order.

Set $G=A*_C B$. Without loss of generality, we can assume that $A$ and $B$ are amenable. There is a surjective map $f$ from $G$ onto the  direct product $A/N \times B/M $. The kernel $K$ of $f$ is the Bass-Serre fundamental group of a graph of groups in which the vertex stabilizers are isomorphic to either $N$ or $M$ and in particular, have bounded size. In this situation, \cite[Theorem 7.7]{ScottWall} applies to give that $K$ embeds in the fundamental group of a finite graph of finite groups and is hence, virtually free (possibly, not finitely generated). Choose a free normal subgroup $F$ of finite index in $K$ and form $N= \cap_{g\in G} F^g$. The subgroup $N$ is normal in $G$ and $K/N$ is locally finite because it is a subdirect product of isomorphic finite groups.  This implies, $G$ is an extension of the free subgroup $N$ by the amenable group $K/N$-by-($A/N \times B/M $) and therefore, is residually amenable. 
\end{proof}
Finally we note that Question \ref{q1} has a positive answer when $C$ is Hausdorff in the profinite topologies of $A$ and $B$. 
\begin{theorem}
 Let $A, B$ be residually amenable groups and let $C$ be a finite subgroup of $A$ and  $B$. Suppose that there are finite index subgroups $A_1 \leq A$ and $B_1 \leq B$ such that $C \cap A_1= 1=C \cap B_1$. Then $A*_C B$ is residually amenable. 
\end{theorem}

\begin{proof} Set $G=A*_C B$. Without loss of generality, assume that $A$ and $B$ are amenable. 
By replacing $A_1$ and $B_1$ with smaller subgroups if necessary we may assume that $A_1 \vartriangleleft A$ and $B_1 \vartriangleleft B$. The group $C$ embeds in both $A/A_1$ and in $B/B_1$.
The amalgam of finite groups $(A/A_1) *_C (B/B_1)$ is residually finite and so has a finite image $P$ where $C$ maps injectively. 
Clearly $G$ maps onto $P$ and the kernel of this map is a finite index normal subgroup $N$ of $G$ which does not intersect any conjugate of $C$. 

The subgroup $N$ acts on the Bass Serre tree for $G$ with trivial edge stabilizers and amenable vertex stabilizers. Therefore $N$ is a free product of amenable groups and is residually amenable. Since $N$ has finite index in $G$, it follows that $G$ is residually amenable. 
\end{proof}

The above result  is applicable for many classes of groups. For example, using that finitely generated nilpotent groups are residually finite we easily obtain the following corollary.

Let  $\gamma_n(G)$ denote the $n$-th term of the lower central series of $G$.
\begin{corollary} Let $A$ and $B$ be finitely generated residually amenable groups. Let $C$ be a finite subgroup of $A$ and $B$ such that $C \cap  \gamma_n (A)=1 = C \cap \gamma_m(B)$ for some $m,n \in \mathbb N$. Then $A*_C B$ is residually amenable.
\end{corollary}

\textbf{Acknowledgement} We thank the referee for pointing out the validity of the second sentence of Theorem 1 and for many suggestions which improved the presentation.

\end{document}